\documentclass{amsart}
\makeatletter
\def\LaTeX{\leavevmode L\raise.42ex
   \hbox{\kern-.3em\size{\sf@size}{0pt}\selectfont A}\kern-.15em\TeX}
\makeatother

\newcommand{\BibTeX}{{\rm B\kern-.05em{\sc
i\kern-.025emb}\kern-.08em\TeX}}

\newtheorem{thm}{Theorem}[section]
\newtheorem{lem}[thm]{Lemma}

\theoremstyle{definition}
\newtheorem{dfn}{Definition}

\numberwithin{equation}{section}

\begin{document}

\title[Deconvolution of band limited
functions]{Deconvolution of band limited
functions on non-compact symmetric spaces}

\author{Isaac Pesenson}
\address{Department of Mathematics, Temple University,
Philadelphia, PA 19122} \email{pesenson@math.temple.edu}


\begin{abstract}It is shown that a band limited
function on a non-compact symmetric space can be reconstructed in
a stable way from some countable sets of values of its convolution
with certain distributions of compact support. A reconstruction
method in terms of frames is given  which is a generalization of
the classical result of Duffin-Schaeffer about exponential frames
on intervals. The second reconstruction method is given in terms
of polyharmonic average splines.
\end{abstract}

\maketitle

\section{Introduction}

 One of the most
interesting properties of the so called band limited functions, i.
e. functions whose Fourier transform has compact support,  is that
they are uniquely determined by their values on some countable
sets of points and can be reconstructed from such values in a
stable way. The sampling problem for band limited functions had
attracted attention of many mathematicians [3], [4], [5], [14].

The mathematical theory of reconstruction of band limited
functions from discrete sets of samples was introduced to the
world of signal analysis and information theory by  Shannon [21].
After Shannon the concept of band limitedness and the Sampling
Theorem became the theoretical foundation of the signal analysis.
 In the classical signal analysis it is assumed that
a signal ($\equiv$ a function) propagates in a Euclidean space.
The most common method to receive, to store and to submit signals
is through  sampling using not the entire signal but rather a
discrete set of its amplitudes measured on a sufficiently dense
set of points.

In the present paper we consider non-compact symmetric spaces
which include  hyperbolic spaces. A sampling Theory on Riemannian
manifolds and in Hilbert spaces was started by the author in [16]-
[20].
 It was shown in [16] that for properly defined band limited functions
 on manifolds the following sampling property holds true: they can
 be recovered from countable sets of their amplitudes measured on
 a sufficiently dense set of points.

Recently  A. Kempf [12], [13] explored our results [16], [18] to
developed a theory which approaches quantization of space-time and
information through the sampling ideas on manifolds.

By using Harmonic Analysis on symmetric spaces we give further
development of the  theory of band limited functions on manifolds.
It is shown that under certain assumptions a band limited function
is uniquely determined and can be reconstructed in a stable way
from some countable sets of values of its convolution with a
distribution.

 The problem of
deconvolution is very natural in the context of Signal Analysis
since convolution is used for denoising a signal with subsequent
reconstruction of the band limited component. On the other hand,
since it is impossible to perform an exact measurement of a signal
at a point, convolution with a distribution can represent a
measuring device.

 As an illustration of the main result it is shown in  Section
 7 that every band limited  function is uniquely determined and can be
reconstructed in a stable way from a set of averages of its
"derivatives" $\Delta^{n}f, n\in \mathbb{N}\bigcup \{0\},$ over a
set of small spheres of a fixed radius whose centers form a
sufficiently dense set, here $\Delta$ is the Laplace-Beltrami
operator of $X$. As a particular case we obtain that every band
limited function on a non-compact symmetric space of rank one is
uniquely determined and can be reconstructed in a stable way from
a set of samples of its "derivatives" $\Delta^{n}f(x_{i}), n\in
\mathbb{N}\bigcup \{0\},$ for a countable and sufficiently dense
set of points $\{x_{i}\}$. Another particular case is, that a band
limited function $f$ on  $X$ is uniquely defined and can be
reconstructed in a stable way from a set of its average values
over a set of small spheres whose centers form a sufficiently
dense set.

The main result of  Section 7 is the Theorem 7.3 which shows that
for all these results the density of centers of the spheres
depends just of the band width $\omega$ and independent of the
radius of the spheres $\tau$ and of the  smoothness index $n$. The
radius of the spheres $\tau$ depends on $\omega$ and smoothness
$n$ but it is independent of the distance $r$ between centers and
can be relatively large compare to $r$. In other words, a stable
reconstruction can still take place even if spheres intersect each
other.

Reconstruction of band limited functions in $L_{2}(R^{d})$ from
their average values over sets of full measure were initiated  in
the paper [6].  Our approach  to this problem is very different
from the approach of [6] and most of our results are new even in
the one-dimensional case. It seems that reconstruction from
averages over sets of measure zero (other then points) was never
considered for band limited functions on $\mathbb{R}^{d}$.

 Let $G$ be a semi-simple
Lie group with a finite center and $K$ its maximal compact
subgroup. We consider a non-compact symmetric space $X=G/K$. We
say that a function $f$ from $L_{2}(X)$ is $\omega$-band limited
if its Helgason-Fourier transform $\hat{f}(\lambda,b)$ is zero for
$<\lambda,\lambda>^{1/2}=\|\lambda\|>\omega$, where $<.,.>$ is the
Killing form. We discuss some properties of such functions in the
Section 3.
 It is shown in particular, that for any
$\omega>0$ and any ball $U\subset X$ restrictions of all
$\omega$-band limited functions to the ball $U$ are dense in
$L_{2}(U, dx)$.

In  Section 4 we prove  uniqueness and stability theorems. In
Section 5 a reconstruction algorithm in terms of frames is
presented. This result is a  generalization of the
 results of Duffin-Schaeffer [5]
about exponential frames on intervals. Let us recall, that the
main result of [5] states, that for so called uniformly dense
sequences of scalars $\{x_{j}\}$ the exponentials
$\{e^{ix_{j}\xi}\}$ form a frame in $L_{2}(I)$ on a certain
interval $I$ of the real line. In other words, if the support of
the Fourier transform of a function $f\in L_{2}(R)$ belongs to
$I$, then $f$ is uniquely determined and can be reconstructed in a
stable way from its values on a countable set of points
$\{x_{j}\}$.

In  Section 6 we consider another way of reconstruction by using
spline-like functions on $X$. It is a far going generalization of
some results of Schoenberg [22].

\section{Harmonic Analysis on symmetric spaces}

A Riemannian symmetric space $X$ is defined as $G/K$, where $G$ is
a connected non-compact semi-simple group Lie with Lie algebra
with finite center and $K$ its maximal compact subgroup. Their Lie
algebras will be denoted respectively as $\textbf{g}$ and
$\textbf{k}$. The group $G$ acts on $X$ by left translations and
it has the "origin" $o=eK$, where $e$ is the identity in $G$.
Every such $G$ admits  Iwasawa decomposition $G=NAK$, where
nilpotent Lie group $N$ and abelian group $A$ have Lie algebras
$\textbf{n}$ and $\textbf{a}$ respectively. The dimension of
$\textbf{a}$ is known as the rank of $X$.
 Letter $M$ is usually used to denote the centralizer of $A$ in
$K$ and letter $B$ is commonly used for the factor $B=K/M$.

Let $\textbf{a}^{*}$ be the real dual of $\textbf{a}$ and $W$ be
the Weyl's group. The $\Sigma$ will be the set of all bounded
roots, and $\Sigma^{+}$ will be the set of all positive bounded
roots. The notation $\textbf{a}^{+}$ has the following meaning
$$
\textbf{a}^{+}=\{h\in \textbf{a}|\alpha(h)>0, \alpha\in
\Sigma^{+}\}
$$
 and is known as positive Weyl's chamber. Let $\rho\in
 \textbf{a}^{*}$ is defined in a way that $2\rho$ is the sum of
 all  positive bounded roots. The Killing form $<,>$ on $\textbf{g}$
 defines a
 metric on $\textbf{a}$. By duality it defines a scalar product on
 $\textbf{a}^{*}$. The $\textbf{a}^{*}_{+}$ is the set of
 $\lambda\in \textbf{a}^{*}$, whose dual belongs to
 $\textbf{a}^{+}$.
According to Iwasawa decomposition for every $g\in G$ there exists
a unique $A(g)\in \textbf{a}$ such that
$$g=n \exp A(g) k, k\in K, n\in N,
$$
 where $\exp :\textbf{a}\rightarrow A$ is the exponential map of
 the
 Lie algebra $\textbf{a}$ to Lie group $A$. On the direct product
 $X\times B$ we introduce function with values in $\textbf{a}$
 using the formula
 \begin{equation}
 A(x,b)=A(u^{-1}g)
 \end{equation}
 where $x=gK, g\in G, b=uM, u\in K$.

For every $f\in C_{0}^{\infty}(X)$ the Helgason-Fourier transform
is defined by the formula
$$
\hat{f}(\lambda,b)=\int_{X}f(x)e^{(-i\lambda+\rho)A(x,b))}dx,
$$
where $ \lambda\in \textbf{a}^{*}, b\in B=K/M, $ and  $dx$ is a
$G$-invariant measure on $X$. This integral can also be expressed
as an integral over group $G$. Namely, if $b=uM,u\in K$, then
\begin{equation}
\hat{f}(\lambda,b)=\int_{G}f(x)e^{(-i\lambda+\rho)A(u^{-1}g))}dg.
\end{equation}
The following inversion formula holds true
$$
f(x)=w^{-1}\int_{\textbf{a}^{*}\times
B}\hat{f}(\lambda,b)e^{(i\lambda+\rho)(A(x,b))}|c(\lambda)|^{-2}d\lambda
db,
$$
where $w$ is the order of the Weyl's group and $c(\lambda)$ is the
Harish-Chandra's function, $d\lambda$ is the Euclidean measure on
$\textbf{a}^{*}$ and $db$ is the normalized $K$-invariant measure
on $B$. This transform can be extended to an isomorphism between
spaces $L_{2}(X,dx)$ and $L_{2}(\textbf{a}^{*}_{+}\times B,
|c(\lambda)|^{-2}d\lambda db)$ and the Plancherel formula holds
true
$$
\|f\|=\left( \int_{\textbf{a}^{*}_{+}\times B}|\hat{f}
(\lambda,b)|^{2}|c(\lambda)|^{-2}d\lambda db\right)^{1/2}.
$$

An analog of the Paley-Wiener Theorem is known which says in
particular that a Helgason-Fourier transform of a compactly
supported distribution is a function which is analytic in
$\lambda$.

To introduce convolutions on $X$ we will need the notion of the
spherical Fourier transform on the group $G$.

For a $\lambda\in \textbf{a}^{*}$ a zonal spherical function
$\varphi_{\lambda}$ on the group $G$ is introduced by the
Harish-Chandra's formula
\begin{equation}
\varphi_{\lambda}(g)=\int_{K}e^{(i\lambda+\rho)(A(kg))}dk.
\end{equation}
If $f$ is a smooth bi-invariant function on $G$ with compact
support its spherical Fourier transform is  a function on
$\textbf{a}^{*}$ which is defined by  the formula
$$
\widehat{f}(\lambda)=\int_{G}f(g)\varphi_{-\lambda}(g)dg.
$$
The inversion formula is
$$
f(g)=w^{-1}\int_{\textbf{a}^{*}}\widehat{f}(\lambda)\varphi_{\lambda}(g)dg.
$$
The corresponding Plancherel formula has the form
$$
\left(\int_{G}|f(g)|^{2}dg\right)^{1/2}=
\left(\int_{\textbf{a}^{*}_{+}}|\widehat{f}(\lambda)|^{2}d\lambda\right)^{1/2}.
$$

If $f$ is a function on the symmetric space $X$ and $\psi$ is a
K-bi-invariant function on $G$ their convolution is a function on
$X$ which is defined by the formula
$$
f\ast \psi(g\cdot o)=\int_{G}f(gh^{-1}\cdot o)\psi(h)dh, g\in G.
$$
By using duality arguments the last definition can be extended
from functions to distributions. It is known, that
\begin{equation}
\widehat{f\ast\psi}(\lambda,b)=\widehat{f}(\lambda,b)\widehat{\psi}(\lambda).
\end{equation}

Denote by $T_{x}(M)$ the tangent space of $M$ at a point $x\in M$
and let $ exp_ {x} $ :  $T_{x}(M)\rightarrow M$ be the exponential
geodesic map i.  e. $exp_{x}(u)=\gamma (1), u\in T_{x}(M)$ where
$\gamma (t)$ is the geodesic starting at $x$ with the initial
vector $u$ :  $\gamma (0)=x , \frac{d\gamma (0)}{dt}=u.$ In what
follows we assume that local coordinates  are defined by $exp$.

By using a uniformly bounded partition of unity
$\{\varphi_{\nu}\}$ subordinate to a cover of  $X$ of finite
multiplicity
$$
X=\bigcup_{\nu} B(x_{\nu}, r),
$$
where $B(x_{\nu}, r)$ is a metric ball at  $x_{\nu}\in X$ of
radius $r$ we introduce Sobolev space $H^{\sigma}(X), \sigma>0,$
as the completion of $C_{0}^{\infty}(X)$ with respect to the norm
\begin{equation}
\|f\|_{H^{\sigma}(X)}=\left(\sum_{\nu}\|\varphi_{\nu}f\|^{2}
_{H^{\sigma}(B(y_{\nu}, r))}\right) ^{1/2}.
\end{equation}

 The usual embedding Theorems for the spaces $H^{\sigma}(X)$
hold true.

The Killing form on $G$ induces an inner product on tangent spaces
of $X$. Using this inner product it is possible to construct
$G$-invariant Riemannian structure on $X$. The Laplace-Beltrami
operator of this Riemannian structure is denoted as $\Delta$.

If the space $X$ has rank one, then in the polar geodesic
coordinate system $(r,\theta_{1},...,\theta_{d-1})$ on $X$ at
every point $x\in X$ it has the form [9]
$$
\Delta=\partial^{2}_{r}+\frac{1}{S(r)}\frac{dS(r)}{dr}\partial
_{r}+\Delta_{S},
$$
where $\Delta_{S} $ is the Laplace-Beltrami operator on the sphere
$S(x,r)$ of the induced Riemannian structure on $S(x,r)$ and
$S(r)$ is the surface area of a sphere of radius $r$ which depends
just on $r$ and is given by the formula
\begin{equation}
S(r)=\Omega_{d}2^{-q}c^{-p-q}sh^{p}(cr)sh^{q}(2cr),
\end{equation}
where $d=dim X=p+q+1, c=(2p+8q)^{-1/2}$, $p$ and $q$ depend on $X$
and
$$
\Omega_{d}=\frac{2\pi^{d/2}}{\Gamma(d/2)}
$$
 is the surface area of the unit sphere in $d$-dimensional
 Euclidean space.

In particular, if a function $f$ is zonal, i.e. depends just on
distance $r$ from the origin $o$,  we have
\begin{equation}
\Delta f(r)=
\frac{1}{S(r)}\frac{d}{dr}\left(S(r)\frac{df(r)}{dr}\right).
\end{equation}

Since for a smooth function $f$ with compact support the function
$\Delta f$ can be expressed as a convolution of $f$ with the
distribution $\Delta \delta_{e}$, where the Dirac measure
$\delta_{e}$ is supported at the identity $e$ of $G$, the formula
(2.4) implies

\begin{equation}
\widehat{\Delta
f}(\lambda,b)=-(\|\lambda\|^{2}+\|\rho\|^{2})\hat{f}(\lambda,b),
f\in C_{0}^{\infty}(X),
\end{equation}
where $\|\lambda\|^{2}=<\lambda,\lambda>,\|\rho\|^{2}=<\rho,\rho>,
<,>$ is the Killing form.

 It is known that for a general $X$ the operator $(-\Delta)$ is a self-adjoint positive
definite operator in the corresponding space $L_{2}(X,dx),$ where
$dx$ is the $G$-invariant measure. The regularity Theorem for the
Laplace-Beltrami operator $\Delta$ states that domains of the
powers
 $(-\Delta)^{\sigma/2}$ coincide with the Sobolev spaces
$H^{\sigma}(X)$ and  the norm (2.5) is equivalent to the graph
norm $\|f\|+\|(-\Delta)^{\sigma/2}f\|$.

\bigskip

  We consider
a ball $B(o,r/4)$ in the invariant metric on $X$. Now we choose
such elements $g_{\nu}\in G$ that the family of balls
$B(x_{\nu},r/4), x_{\nu}=g_{\nu}\cdot o,$ has the following
maximal property: there is no ball in $X$ of radius $r/4$ which
would have empty intersection with every ball from this family.
Then the balls of double radius $B(x_{\nu},r/2)$ would form a
cover of $X$. Of course, the balls $B(x_{\nu},r)$ will also form a
cover of $X$. Let us estimate the multiplicity of this cover.

Note, that the Riemannian volume $B(\rho)$ of a ball of radius
$\rho$ in $X$ is independent of its  center and   is given by the
formula
$$
B(\rho)=\int_{0}^{\rho}S(t)dt,
$$
where the surface area $S(t)$ of any sphere of radius $t$ is given
by the formula (2.4).

 Every ball from the family $\{B(x_{\nu}, r)\}$, that has
non-empty intersection with a particular ball $B(x_{j}, r)$ is
contained in the ball $B(x_{j}, 3r)$. Since any two balls from the
family $\{B(x_{\nu}, r/4)\}$ are disjoint, it gives the following
estimate for the index of multiplicity  of the cover $\{B(x_{\nu},
r)\}$:
\begin{equation}
\frac{B(3r)}{B(r/4)}=\frac{\int_{0}^{3r}S(t) dt}{\int_{0}^{r/4}
S(t) dt}.
\end{equation}

It is clear that for all sufficiently small $r>0$ this fraction is
bounded. In what follows we we will use the notation
$$
N=\sup_{0<r<1}\frac{B(3r)}{B(r/4)}.
$$

 So, we proved the following Lemma.
\begin{lem}
For any sufficiently small  $r>0$ there
 exists a set of points $\{x_{\mu}\}$ from $X$ such that

1) balls $B(x_{\mu}, r/4)$ are disjoint,

2) balls $B(x_{\mu}, r/2)$ form a cover of $X$,

3) multiplicity of the cover by balls $B(x_{\mu}, r)$ is not
greater $N.$

\end{lem}

We will use notation $Z(\{x_{\mu}\}, r, N)$ for any set of points
$\{x_{\mu}\}\in X$ which satisfies the properties 1)- 3) from the
last Lemma and we will call such set a metric $(r,N)$-lattice of
$X$.

The following results can be found in [16] for any homogeneous
manifold $X$.
\begin{sloppypar}
\begin{thm}
For any $k>d/2$  there exist  constants $C=C(X, k, N)>0,
r_{0}(X,k,N),$
 such that for any $0<r<r_{0}$ and any $(r,N)$-lattice
 $Z=Z(\{x_{\mu}\},r,N)$ the following inequality holds true

\begin{equation}
\|f\|\leq C\left\{r^{d/2}\left(\sum_{x_{j}\in Z}
|f(x_{j})|^{2}\right)^{1/2}+r^{k}\|\Delta^{k/2}f\|\right\}, k>d/2.
\end{equation}
and there exists a constant $C_{1}=C_{1}(X, k, N)$ such that

$$\left(\sum _{x_{j}\in Z}|f(x_{j})|^{2}\right)^{1/2}\leq
 C_{1}\|f\|_{H^{k}(X)}, f\in H^{k}(X).
 $$
\end{thm}
\end{sloppypar}

\section{Band limited functions}

\begin{dfn}
We will say that $f\in L_{2}(X,dx)$ belongs to the class
$B_{\omega}(X)$ if its Helgason-Fourier transform has compact
support in the sense that $\hat{f}(\lambda,b)=0$ for
$\|\lambda\|>\omega.$ Such functions will be called $\omega$-band
limited.
\end{dfn}

We have the following important result, which is a specification
of some results in [15]- [18].
\begin{thm}
A function $f$ belongs to $B_{\omega}(X)$ if and only if the
following Bernstein inequality holds true for all $\sigma >0$,
\begin{equation}
\|\Delta^{\sigma}f\|\leq (\omega^{2}+\|\rho\|^{2})^{\sigma}\|f\|.
\end{equation}

\end{thm}
\begin{proof}
 By using the Plancherel formula and (2.8) we obtain that for every $\omega$-
  band limited function
  $$
\|\Delta^{\sigma}f\|^{2}=\int_{\|\lambda\|<\omega}\int_{B}(\|\lambda\|^{2}+\|\rho\|^{2})
^{\sigma} |\widehat{f}(\lambda,b)|^{2}|c(\lambda)|^{-2}d\lambda
db\leq
$$
$$
 (\omega^{2}+\|\rho\|^{2})^{\sigma}
\int_{\textbf{a}^{*}}\int_{B}|\widehat{f}(\lambda,b)|^{2}|c(\lambda)|^{-2}d\lambda
db=(\omega^{2}+\|\rho\|^{2})^{\sigma}\|f\|^{2}.
$$

  Conversely, if $f$ satisfies (3.1), then for any $\varepsilon>0$
  and any $\sigma >0$ we have
  $$
  \int_{\|\lambda\|\geq\omega+\varepsilon}\int_{B}
  |\hat{f}(\lambda,b)|^{2}|c(\lambda)|^{-2}d\lambda db\leq
  $$
  $$
  \int_{\|\lambda\|\geq \omega+\varepsilon}\int_{B}
  (\|\lambda\|^{2}+\|\rho\|^{2})^{-2\sigma}
 ( \|\lambda\|^{2}+\|\rho\|^{2})^{2\sigma}|\hat{f}(\lambda,b)|^{2}|c(\lambda)|
 ^{-2}d\lambda db\leq
  $$
  \begin{equation}
  \left(\frac{\omega^{2}+\|\rho\|^{2}}
  {(\omega+\varepsilon)^{2}+\|\rho\|^{2}}\right)^{2\sigma}
  \|f\|^{2}.
  \end{equation}
It means, that for any $\varepsilon>0$ the function
$\widehat{f}(\lambda,b)$ is zero on $\{\lambda: \|\lambda\|\geq
\omega+\varepsilon\}\times B$. The statement is proved.
\end{proof}

Now we are going to prove the following "density"  result.
\begin{thm}
For every $\omega>0$ and every ball $U\subset X$ restrictions to
$U$ of all functions from $B_{\omega}(X)$ are dense in the space
$L_{2}(U, dx)$.
\end{thm}
\begin{proof}
 Indeed, assume that $\psi\in
L_{2}(U, dx)$ is a function which is orthogonal to all
restrictions to $U$ of all functions from $B_{\omega}(X)$. We
extend $\psi$ by zero outside of $U$. By the Paley-Wiener Theorem
the Helgason-Fourier transform $\hat{\psi}(\lambda,b)$ is
holomorphic in $\lambda$ and at the same time should be orthogonal
to all functions from $L_{2}\left(\mathcal{B}(0,\omega)\times B;
|c(\lambda)|^{-2}d\lambda db\right)$, where
$$
\mathcal{B}(0,\omega)=\{\lambda \in \textbf{a}^{*}:
\|\lambda\|\leq \omega\}.
$$
It implies that $\hat{\psi}$ is zero. Consequently, the function
$\psi$ is zero. It proves the Theorem.
\end{proof}

\section{Uniqueness and stability}

We say that set of points $M=\{x_{j}\}$ is a uniqueness set for
$B_{\omega}(X)$, if every $f\in B_{\omega}(X)$ is uniquely
determined by its values on $M$.

\begin{thm}
If a set $M=\{x_{j}\}$ is a uniqueness set for the space
$B_{\omega}(X)$, then for any bi-invariant distribution of compact
support $\phi$ every function $f\in B_{\omega}(X)$ is uniquely
determined by the set of values $f\ast\phi(x_{j})$.
\end{thm}
\begin{proof}
If $f\in B_{\omega}(X)$ then because
$$
\widehat{f\ast\phi}=\widehat{f}\widehat{\phi}
$$
the function $f\ast\phi$ also belongs to $B_{\omega}(X)$ and by
assumption is uniquely determined by its values
$f\ast\phi(x_{j})$. But by the Paley-Wiener Theorem the zero set
of the function $\widehat{\phi}$ has measure zero. Thus, the
equality
$$
\widehat{f_{1}\ast\phi}-\widehat{f_{2}\ast\phi}=
\widehat{\phi}(\widehat{f_{1}}-\widehat{f_{2}})\equiv 0,
$$
implies that $f_{1}=f_{2}$ as $L_{2}$-functions. The statement is
proved.
\end{proof}

For any uniqueness set $M$ and any $\omega>0$ the notation
$l_{2}^{\omega}(M)$ will be used for a linear subspace of all
sequences $\{v_{j}\}$ in $l_{2}$ for which there exists a function
$f$ in $B_{\omega}(X)$ such that
$$
f\ast\phi(x_{j})=v_{j}, x_{j}\in M.
$$
In general $l_{2}^{\omega}(M)\neq l_{2}$.

\begin{dfn}
\textit{A linear reconstruction method $R$} from a uniqueness set
$M$ is a linear operator
$$
R:l_{2}^{\omega}(M)\rightarrow B_{\omega}(X)
$$
such that
$$
R: \{f\ast\phi(x_{j})\}\rightarrow f.
$$

 The reconstruction method is said to be stable, if it is
continuous in topologies induced respectively by $l_{2}$ and
$L_{2}(X)$.
\end{dfn}

The  Theorem 4.1 does not guarantee stability but the next one
does.

\begin{sloppypar}
\begin{thm}
There exists a constant $c=c(X,N)>0$ such
that for any $\omega>0,$ any $(r,N)$-lattice
$Z(\{x_{\mu}\},r,N)$ with
$$
0<r<c(\omega^{2}+\|\rho\|^{2})^{-1/2},
$$
and every bi-invariant distribution $\phi$ with compact support,
whose Helgason-Fourier transform $\widehat{\phi}$ does not have
zeros on $\mathcal{B}(0,\omega),$ every function $f\in
B_{\omega}(X)$ is uniquely determined and   reconstruction method
from the set of samples $f\ast\phi(x_{j})$ is stable.
\end{thm}
\end{sloppypar}
\begin{proof}
The formulas (2.10) and (3.1) imply, that there exist constants
$C,c$ such that for any $\omega>0$, any $Z=Z(\{x_{\mu}\}, r, N)$
with
$$
0<r<c(\omega^{2}+\|\rho\|^{2})^{-1/2}
$$
and every $f\in
B_{\omega}(X)$
\begin{equation}
\|f\|\leq Cr^{-d/2}\left(\sum  _{x_{j}\in Z}
|f(x_{j})|^{2}\right)^{1/2}.
\end{equation}
Indeed, by the Theorem 2.2 and the Theorem 3.1 we have
\begin{equation}
\|f\|\leq C\left\{r^{d/2}\left(\sum_{x_{j}\in Z}
|f(x_{j})|^{2}\right)^{1/2}+r^{k}(\omega^{2}+\|\rho\|^{2})^{k/2}\|f\|\right\},
\end{equation}
where $k>d/2$. If we  choose
$$
0<r<c(\omega^{2}+\|\rho\|^{2})^{-1/2}, c=C^{-1},
$$
then the inequality (4.2) gives the  inequality (4.1) which
implies uniqueness. Next, by applying the Plancherel Theorem and
the inequality (4.1) to the function $f\ast\phi$ we obtain the
inequality
\begin{equation}
\|f\|=\|\widehat{f}\|=\|\widehat{\phi}^{-1}\widehat{f}\widehat{\phi}\|\leq
C_{1}\|\widehat {f\ast\phi}\|\leq
$$
$$
Cr^{-d/2}\left(\sum_{x_{j}\in
Z(x_{j},r,N)}|f\ast\phi(x_{j})|^{2}\right)^{1/2},
\end{equation}
where
$$
C_{1}=\left(\inf_{\lambda\in \mathcal{B}(0,\omega)}
\widehat{\phi}(\lambda)\right)^{-1}.
$$
This inequality  (4.3) implies  stability of the reconstruction
method from the samples $f\ast\phi(x_{j})$. The Theorem is proved.
\end{proof}

\section{Reconstruction in terms of frames}

In  this section we are going to present a way of reconstruction
of a function $f\in B_{\omega}(X)$ from a set of samples of its
convolution $f\ast\phi$ in terms of frames in Hilbert spaces [1],
[5].

If $\delta_{x_{j}}$ is a Dirac distribution at a point $x_{j}\in
X$ then according to the inversion formula for the
Helgason-Fourier transform we have
$$
\left<\delta_{x_{j}},f\right>= w^{-1}\int_{\textbf{a}^{*}\times
B}\hat{f}(\lambda,b)e^{(i\lambda+\rho)(A(x_{j},b))}|c(\lambda)|^{-2}d\lambda
db.
$$

It implies that if $f\in L_{2}(X)$ then the action on
$\hat{f}(\lambda, b)$ of the Helgason-Fourier transform $
\widehat{\delta_{x_{j}}}$ of $ \delta_{x_{j}}$ is given by the
formula
\begin{equation}
\hat{f}(\lambda, b) \rightarrow \left<\widehat{\delta_{x_{j}}},
\hat{f}\right> =w^{-1}\int_{\textbf{a}^{*}\times
B}e^{(i\lambda+\rho)(A(x_{j},b))}\hat{f}(\lambda,
b)|c(\lambda)|^{-2}d\lambda db.
\end{equation}
\begin{thm}
There exists a constant $ c=c(X,N)$ such that for any given
$\omega>0$, for every $(r,N)$-lattice  $Z=Z(\{x_{\mu}\},r, N)$
with
$$
0<r<c(\omega^{2}+\|\rho\|^{2})^{-1/2},
$$
the following statements hold true.

1)  The set of functions
 $\{\widehat{\delta_{x_{j}}\}}$ forms a frame in the space
\begin{equation}
L_{2}\left(\mathcal{B}(0,\omega)\times B;
|c(\lambda)|^{-2}d\lambda db\right)
\end{equation}
 and there exists a  frame
 $\{\Theta_{j}\}$ in the space $B_{\omega}(X)$
such that every $\omega$-band limited function $f\in
B_{\omega}(X)$ can be reconstructed from a set of samples
$\{\delta_{x_{j}}(f)\}$ by using the formula
\begin{equation}
f=\sum_{x_{j}\in Z}\delta_{x_{j}}(f)\Theta_{j}.
\end{equation}

2) If the Helgason-Fourier transform of a compactly supported
distribution $\phi$ does not have zeros on the ball
$\mathcal{B}(0,\omega)$, then the Helgason-Fourier transforms of
the distributions
$$
 f\rightarrow f\ast\phi(x_{j})
 $$
 form a frame in the space $
L_{2}\left(\mathcal{B}(0,\omega)\times B;
|c(\lambda)|^{-2}d\lambda db\right)$ and
 there exists a frame $\{\Phi_{j}\}$
in the space $B_{\omega}(X)$ such, that every $f\in B_{\omega}(X)$
can be reconstructed from the samples of the convolution
$f\ast\phi$ by using the formula

\begin{equation}
f= \sum_{x_{j}\in Z}\delta_{x_{j}}(f\ast\phi)\Phi_{j}.
\end{equation}
\end{thm}

\begin{proof}

An application of the Plancherel formula along with the inequality
(4.3) and the second inequality in the Theorem 2.2 give that for
any $f\in B_{\omega}(X)$

\begin{equation}
C_{1}\|\widehat{f}\|_{\Lambda_{2}}\leq \left(\sum_{x_{j}\in
Z}|<\widehat{\delta_{x_{j}}},\widehat{f}>|^{2}\right)^{1/2}\leq
C_{2}\|\widehat{f}\|_{\Lambda_{2}},
\end{equation}
where $<.,.>$ is the scalar product in the space $\Lambda_{2}=
L_{2}\left(\mathcal{B}(0,\omega)\times B;
|c(\lambda)|^{-2}d\lambda db\right)$. The first statement of the
Theorem is just another interpretation of the inequalities (5.5).

We  consider the so called frame operator
$$
F(\widehat{f})=\sum_{j}<\widehat{\delta_{x_{j}}},\widehat{f}>\widehat{\delta_{x_{j}}}.
$$
It is known that the operator $F$ is invertible and the formula
\begin{equation}
\widehat{\Theta_{j}}=F^{-1}\widehat{\delta_{x_{j}}}.
\end{equation} gives a dual frame $\widehat{\Theta_{j}}$ in
 $ L_{2}\left(\mathcal{B}(0,\omega)\times B;
|c(\lambda)|^{-2}d\lambda db\right)$. A reconstruction formula of
a function $f$ can be written in terms of the dual frame as
\begin{equation}
\widehat{f}=\sum_{j}<\widehat{\Theta_{j}},\widehat{f}>\widehat{\delta_{x_{j}}}=
\sum_{j}<\widehat{\delta_{x_{j}}},\widehat{f}>\widehat{\Theta_{j}},
\end{equation}
where inner product is taken in the space $
L_{2}\left(\mathcal{B}(0,\omega)\times B;
|c(\lambda)|^{-2}d\lambda db\right)$.

Taking the Helgason-Fourier transform of both sides of the last
formula we obtain the formula (5.3) of our Theorem.

The second statement of the Theorem is a consequence of the first
one, of the Plancherel formula and of our assumption that
$\widehat{\phi}$ does not have zeros in $\mathcal{B}(0,\omega)$.
It is clear, that in the formula (5.4) every $\Phi_{j}$ is the
inverse Helgason-Fourier transform of the function
$\widehat{\Theta_{j}}/ \widehat{\phi}$.
\end{proof}

In the classical case when $X$ is the one-dimensional Euclidean
space we have
\begin{equation}
\widehat{\delta_{x_{j}}}(\lambda)= e^{i x_{j}\lambda}.
\end{equation}
In this situation the first statement of the Theorem  means that
the complex exponentials $e^{i x_{j}\lambda}, x_{j}\in
Z(\{x_{\mu}\},r, N )$ form a frame in the space $L_{2}([-\omega,
\omega])$.

Note that in the case of a uniform point-wise sampling in the
space $L_{2}(R)$ this result gives the classical sampling formula
$$
f(t)= \sum f(\gamma n\Omega )\frac{\sin(\omega (t-\gamma n\Omega
))}{\omega (t-\gamma n\Omega )}, \Omega =\pi /\omega, \gamma<1,
$$
with a certain oversampling.

\section{Reconstruction by using polyharmonic splines on $X$}

To present our second method of reconstruction we consider the
following optimization problem. Although the results of this
section are similar to the corresponding results from our paper
[17], the presence of the Helgason-Fourier transform allows to
make all our constructions much more explicit.

\bigskip

 \textsl{For a given $(r,N)$-lattice $Z=Z(\{x_{\mu}\},r,N)$
find a function $L_{\nu}^{k}$ in the Sobolev space $H^{2k}(X)$ for
which $L_{\nu}^{k}(x_{\mu})=\delta_{\nu\mu}, x_{\mu}\in Z$, and
which minimizes the functional $u\rightarrow \| \Delta^{k} u\|$}.
Here the $\delta_{\nu\mu}$ is the Kronecker delta.

\begin{thm}
For a given $(r,N)$-lattice $Z(\{x_{\mu}\},r,N)$ the following
statements are true.

1) The above optimization problem does have a unique solution.

2) For every $L_{\nu}^{k}$ there exists a sequence
$\alpha=\{\alpha_{j,\nu}\}\in l_{2}$ such that
\begin{equation}
\widehat{L_{\nu}^{k}}(\lambda,b)= \sum_{x_{j}\in Z}
\alpha_{j,\nu}\widehat{\delta_{j}}(\lambda,b),
\end{equation}
where $\widehat{\delta_{j}}$ is the Helgason-Fourier transform of
the distribution $\delta_{j}$ and is given by the formula

\begin{equation}
\widehat\delta_{x_{j}}=e^{(-i\lambda+\rho)(A(x_{j},b))}, x_{j}\in
Z.
\end{equation}.
\end{thm}

\begin{proof}
The Theorem 2.2 implies that for a fixed set of values
$\{v_{j}\}\in l_{2}$ the minimum of the functional
$$
u\rightarrow \|\Delta^{k}u\|
$$
 is the same as the minimum of the functional
 $$
 u\rightarrow \|\Delta^{k}u\|+\left
 (\sum_{x_{j}\in Z}|v_{j}|^{2}\right)^{1/2},
$$
with a set of constrains $ u(x_{j})=v_{j}$, $ x_{j}\in
Z(\{x_{\mu}\},r,N)$.

Since the last functional is equivalent to the Sobolev norm it
allows to perform the following procedure.

For the given sequence $v_{j}=\delta_{j,\nu}$ where $ \nu $ is
fixed natural number  consider a function $f$ from $H^{2k}(X)$
such that $f(x_{j})=v_{j}.$ Let $Pf$
 denote the orthogonal projection of this function $f$ (in the Hilbert
space $H^{2k}(X)$ with natural inner product) on the subspace of
functions vanishing on $Z(\{x_{\mu}\},r,N)$. Then the function
$L^{k}_{\nu}=f-Pf$ will be the unique solution of the above
minimization problem.

 The condition  that $L_{\nu}^{k}\in H^{2k}(X)$ is a solution to the
minimization problem implies, that $\Delta^{k}L_{\nu}^{k}$ should
be orthogonal to all functions of the form $\Delta^{k}h,$ where
$h\in H^{2k}(X)$ and has the property $h(x_{\mu})=0$ for all
$\mu$. This leads to a differential equation
$$
\int_{X}(\Delta^{k}L_{\nu}^{k}) \overline{\Delta^{k}h}dx=
\sum_{j}\alpha_{j,\nu} \overline{h}(x_{j}),
$$
for an $l_{2}$- sequence $\{\alpha_{j,\nu}\}$. Thus, in the sense
of distributions
\begin{equation}
\Delta ^{2k}L_{\nu}^{k} =\sum_{j}\alpha_{j,\nu}\delta_{x_{j}}.
\end{equation}

Taking the Helgason-Fourier transform of both sides of (6.3) in
the sense of distributions we obtain the  equation (6.1).

The Theorem is proved.
\end{proof}

We will need the following Lemma.

\begin{lem}
If for some $f\in H^{2\sigma}(X), a,\sigma>0,$
\begin{equation}
\|f\|\leq a\|\Delta^{\sigma}f\|,
\end{equation}
then for the same $f, a, \sigma$ and all $s\geq 0, m=2^{l}, l=0,
1, ...,$
\begin{equation}
\|\Delta^{s}f\|\leq a^{m}\|\Delta^{m\sigma+s}f\|,
\end{equation}
if $f\in H^{2(m\sigma+s)}(X).$

\end{lem}

\begin{proof} In what follows
we use the notation $d\mu(\lambda)$ for the measure
$|c(\lambda)|^{-2}d\lambda$. The Plancherel Theorem allows to
write our assumption (6.4) in the form
$$
\|f\|^{2}=\int_{\textbf{a}_{+}\times
B}|\hat{f}(\lambda,b)|^{2}d\mu(\lambda)db\leq
$$
$$
a^{2}\int_{\textbf{a}_{+}\times
B}(\|\lambda\|^{2}+\|\rho\|^{2})^{2\sigma}|\hat{f}(\lambda,b)|^{2}d\mu(\lambda)db\leq
$$
$$
a^{2}\int_{0<(\|\lambda\|^{2}+\|\rho\|^{2})<a^{-1/\sigma}}\int_{
B}(\|\lambda\|^{2}+\|\rho\|^{2})^{2\sigma}|\hat{f}(\lambda,b)|^{2}d\mu(\lambda)db+
$$
$$
a^{2}\int_{(\|\lambda\|^{2}+\|\rho\|^{2})>a^{-1/\sigma}}\int_{
B}(\|\lambda\|^{2}+\|\rho\|^{2})^{2\sigma}|\hat{f}(\lambda,b)|^{2}d\mu(\lambda)db.
$$

From this we obtain
$$
0\leq\int_{(\|\lambda\|^{2}+\|\rho\|^{2})<a^{-1/\sigma}}\int_{B}
\left(|\hat{f}(\lambda,b)|^{2}
-a^{2}(\|\lambda\|^{2}+\|\rho\|^{2})^{2\sigma}|\hat{f}(\lambda,b)|^{2}\right)
d\mu(\lambda)db\leq
$$
$$
\int_{(\|\lambda\|^{2}+\|\rho\|^{2})>a^{-1/\sigma}}\int_{B}
\left(a^{2}(\|\lambda\|^{2}+\|\rho\|^{2})^{2\sigma}|\hat{f}(\lambda,b)|^{2}-
|\hat{f}(\lambda,b)|^{2}\right) d\mu(\lambda)db.
$$

Multiplication of  this inequality by
$a^{2}(\|\lambda\|^{2}+\|\rho\|^{2})^{2\sigma}$
 will only
 improve the existing inequality
and then using the Plancherel Theorem once again we will obtain

$$\|f\|\leq a\|\Delta^{\sigma}f\|\leq a^{2}\|\Delta^{2\sigma}f\|.
$$

It is now clear that using induction we can prove

$$
\|f\|\leq a^{m}\|\Delta^{m\sigma}f\|, m=2^{l}, l\in \mathbb{N}.
$$

But then, using the same arguments we have for any $s>0$

$$
0\leq \int_{(\|\lambda\|^{2}+\|\rho\|^{2})< a^{-1/\sigma}}\int_{B}
(a^{2s}(\|\lambda\|^{2}+\|\rho\|^{2})^{2s\sigma}
|\hat{f}(\lambda,b)|^{2} -
$$
$$
a^{2(m+s)}(\|\lambda\|^{2}+\|\rho\|^{2})^{2(m+s)\sigma}|
\hat{f}(\lambda,b)|^{2})d\mu(\lambda)db\leq
$$
$$\int_{(\|\lambda\|^{2}+\|\rho\|^{2})> a^{-1/\sigma}}\int_{B}
(a^{2(m+s)}(\|\lambda\|^{2}+\|\rho\|^{2})^{2(m+s)\sigma}
|\hat{f}(\lambda,b)|^{2}-
$$
 $$
 a^{2s}(\|\lambda\|^{2}+\|\rho\|^{2})^{2s
\sigma}|\hat{f}(\lambda,b)|^{2})d\mu(\lambda)db,
$$
that gives the desired inequality (6.5) if $t=\sigma s.$
\end{proof}

\begin{lem}
There exists a constant $C=C(X, N),$ and  for any $\omega>0$ there
exists a $ r_{0}(X, N, \omega)>0$ such that for any $0<r<r_{0}(X,
N, \omega)$, any $(r,N)$-lattice $Z=Z(\{x_{\mu}\},r, N)$, the
following inequality holds true
$$
\|\sum _{j} f(x_{j})L^{2^{l}d+s}_{\nu}-f\|_{H^{s}(X)}\leq
$$
\begin{equation}
 \left(C
r^{2}(\omega^{2}+\|\rho\|^{2})\right)^{2^{l}d}
(\omega^{2}+\|\rho\|)^{s}\|f\|,
\end{equation}
for any $s\geq 0, l=0, 1, ... ,$ and any $f\in B_{\omega}(X)$.

Moreover, if $s>d/2+k$ then

\begin{equation}
\|\sum_{j}f(x_{j})L^{2^{l}d+s}_{j}-f\|_{C_{b}^{k}(X)} \leq
$$
$$
\left(Cr^{2}
(\omega^{2}+\|\rho\|^{2})\right)^{2^{l}d}(\omega^{2}+\|\rho\|^{2})^{s}\|f\|,
l=0, 1, ...
\end{equation}
where $C_{b}^{k}(X)$ the space of $k$ continuously differentiable
bounded functions on $X$.
\end{lem}

\begin{sloppypar}
\begin{proof}
First we will show that there exist constants $C=C(X, N),$ and $
r_{0}(X, N)>0$ such that for any $0<r<r_{0}(X, N)$, any
$(r,N)$-lattice $Z=Z(\{x_{\mu}\},r, N)$,  the following inequality
holds true
\begin{equation}
\|\sum_{j}f(x_{j})L_{j}^{2^{l}d+s}-f\|_{H^{s}(X)}\leq
2(Cr^{2d})^{2^{l}}\|\Delta^{2^{l}d+s}f\|.
\end{equation}
for any $s\geq 0, l=0, 1, ... ,$ and any $f\in H^{s}(X)$.

Indeed, by the Theorem 2.2, since for $k=d$ the function
$\sum_{j}f(x_{j})L_{j}^{2^{l}d+s}$ interpolates $f$ we have
$$
\|\sum_{j}f(x_{j})L_{j}^{2^{l}d+s}-f\|\leq
Cr^{2d}\|\Delta^{d}(\sum_{j}f(x_{j})L_{j}^{2^{l}d+s}-f)\|
$$

By the Lemma 6.2 we obtain

$$
\|\Delta^{s}(\sum_{j}f(x_{j})L_{j}^{2^{l}d+s}-f)\|\leq
(Cr^{2d})^{2^{l}}\|\Delta^{2^{l}d+s}(\sum_{j}f(x_{j})L_{j}^{2^{l}d+s}-f)\|.
$$

Using  the minimization property we obtain (6.7).

 If $s>d/2+k$ then an application of (6.7) and the Sobolev embedding Theorem
 gives
\begin{equation}
\|\sum_{j}f(x_{j})L^{2^{l}d+s}_{j}-f\|_{C_{b}^{k}(X)} \leq
\left(Cr^{2d}\right) ^{2^{l}} \|\Delta^{2^{l}d+s}f\|, l=0, 1, ...
,
\end{equation}
where $C_{b}^{k}(X)$ the space of $k$ continuously differentiable
bounded functions on $X$.

The inequalities (6.7) and (6.8) along with the inequality
$$
\|\Delta^{\sigma}f\|\leq(\omega^{2}+\|\rho\|^{2})^{\sigma}\|f\|,
f\in B_{\omega}(X),
$$
imply the Lemma.
\end{proof}
\end{sloppypar}

Inequalities (6.6) and (6.7) show that a function $f\in
B_{\omega}(X)$ can be reconstructed from its values $f(x_{j})$ as
a limit of interpolating splines
$$
\sum_{j} f\ast\phi(x_{j})L_{j}^{2^{l}d+s}
$$
in Sobolev or uniform norms when $l$ goes to infinity. By using
the Plancherel formula for the Helgason-Fourier transform we
obtain the following reconstruction algorithm in which we assume
that a reciprocal of the Helgason-Fourier transform of a
distribution $\phi$ is defined almost everywhere and bounded.

\begin{thm}
There exists a constant $c=c(X, N),$ so that  for any $\omega>0$,
any $(r,N)$-lattice $Z=Z(\{x_{\mu}\},r, N)$ with
$$
r<\left(c(\omega^{2}+\|\rho\|^{2})\right)^{-1/2},
$$
every function $f\in B_{\omega}(X)$ is uniquely determined by the
values  of $f\ast\phi$ at the points $\{x_{j}\}$ and can be
reconstructed as a limit of
$$
\sum_{j} f\ast\phi(x_{j})\Lambda_{j}^{2^{l}d+s}, l\rightarrow
\infty,
$$
in Sobolev and uniform norms, where
$$
\widehat{\Lambda_{j}^{2^{l}d+s}}=\frac{\widehat{L_{j}^{2^{l}d+s}}}{\widehat{\phi}}.
$$
\end{thm}

\section{Example: Convolution with spherical average distribution}

In this section we assume that the symmetric space $X=G/H$ has
rank one, which means that the algebra Lie $\textbf{a}$ of the
abelian component $A$ in the Iwasawa decomposition $G=NAK$ has
dimension one.
 For  a point $y\in X,$ the $S(y,\tau), \tau>0,$ will
denote the sphere with center $y$ and of radius $\tau$. For a
smooth function $f$ with compact support we define  the spherical
average
$$
(M^{\tau}f)(y)=S(\tau)^{-1}\int_{S(y,\tau)}f(z)ds(z), \tau>0,
$$
where $ds$ is the measure on $S(y,\tau)$ and $S(\tau) $ is its
volume, i.e. the surface integral of $ds$ over $S(\tau)$ which is
independent on the center.

It is clear, that $M^{\tau}f$ is a convolution of $f$ with the
distribution
$$
m^{\tau}(f)=S^{-1}(\tau)\int_{S(o,\tau)}f(x)ds(x), \tau>0, f\in
C_{0}^{\infty}(X).
$$

Since $X$ is a symmetric space of rank one, the group $K$ is
transitive on every $S(o,\tau)$ and this fact allows to reduce
integration over the sphere $S(o,\tau)$ to an integral over group
$K$. Because the invariant  measure $dk$ on $K$ is normalized,it
results in the following change of variables formula
\begin{equation}
(M^{\tau}f)(g\cdot o)=\int_{K}f(gkz_{\tau})dk,
\end{equation}
for any $z_{\tau}\in S(o,\tau),  g\in G$.

Next, by using  the following formula
$$
\int_{X}f(x)dx=\int_{G}f(g\cdot o)dg, f\in C_{0}^{\infty}(X)
$$
along with Minkowski inequality and the formula (7.1) we obtain

$$
\|M^{\tau}(f)\|\leq \left\{ \int_{G}\int_{K}|f(gkz_{\tau})|^{2}dk
dg\right\}^{1/2}\leq \
$$
\begin{equation}
\int_{K}\left\{\int_{G}|f(go)|^{2}dg\right\}^{1/2}dk=\|f\|, f\in
C_{0}^{\infty}(X),
\end{equation}
where we have used the fact that the invariant measure $dk$ is
normalized on $K$.

 Since the set $C_{0}^{\infty}(X)$ is dense in
$L_{2}(X)$ it implies the following.
\begin{lem}
Operator $f\rightarrow M^{\tau}(f)$ has continuous extension from
$C_{0}^{\infty}(X)$ to $L_{2}(X)$ which will be also denoted as
$M^{r}(f)$ and for this extension
$$
\|M^{\tau}(f)\|\leq \|f\|
$$
for all $f\in L_{2}(X), \tau>0$.
\end{lem}

Since the operator $M^{\tau}$ is a convolution with the
distribution $m^{\tau}$, we have
$$
\widehat{M^{\tau}(f)}=\widehat{m^{\tau}}\hat{f}.
$$

Our goal is to consider even more general operator. We choose an
$n\in \mathbb{N}\bigcup\{0\},$ and consider the operator
$$
f\rightarrow M^{\tau}((-\Delta)^{n}f).
$$

It is clear, that this operator is a convolution with the
distribution $(-\Delta)^{n}m^{\tau}$.

The nearest goal is to find the Fourier transform
$\widehat{(-\Delta)^{n}m^{\tau}}$ of the
 distribution $(-\Delta)^{n}m^{\tau}$. Namely, we prove the following.
\begin{lem}
The Helgason-Fourier transform of the distribution
$\widehat{(-\Delta)^{n}m^{\tau}}$ is given by the formula
\begin{equation}
\widehat{(-\Delta)^{n}m^{\tau}}=(\|\lambda\|^{2}+\|\rho\|^{2})^{n}\varphi_{\lambda}(\exp
\tau V),
\end{equation}
where $V$ is a vector from algebra Lie $\mathbf{a}$ and
$\varphi_{\lambda}$ is the zonal spherical function i.e.

$$
\varphi_{\lambda}(g)=\int_{K}e^{(i\lambda+\rho)(A(kg))}dk, g\in G.
$$
\end{lem}
\begin{proof}

Since the  operator $(-\Delta)^{n}, n\in \mathbb{N}\bigcup\{0\},$
can be represented as a convolution with the distribution
$\Delta^{n}\delta_{e}, $ where $\delta_{e}$ has support at the
identity $e\in G$, it is clear, that
$$
\widehat{(-\Delta)^{n}m^{\tau}}=(\|\lambda\|^{2}+\|\rho\|^{2})^{n}\widehat{m^{\tau}}.
$$
But
$$
\widehat{m^{\tau}}(\lambda,b)=m^{\tau}(e^{(i\lambda+\rho)A(x,b)})=
$$
\begin{equation}
S^{-1}(\tau)\int_{S(o,\tau)}e^{(i\lambda+\rho)A(x,b)}ds(x),
\tau>0.
\end{equation}

By the formula (7.1) and by the Harish-Chandra's formula (2.3) the
last integral is
$$
S^{-1}(\tau)\int_{S(o,\tau)}e^{(i\lambda+\rho)A(x,b)}ds(x)=
\int_{K}e^{(-i\lambda+\rho)A(k^{-1}g)}dk=\varphi_{\lambda}(g),
$$
where $x=gK, b=kM\in B=K/M$.

 Because the symmetric space $X=G/K$, has rank one, there exists a unique vector $V$
  that belongs to the Lie algebra
 $\textbf{a}$, such that the curve  $\exp(- \tau V)\cdot o$ is a geodesic  in
$X$ and the one-parameter group $\exp(- \tau V)$ moves the origin
$o\in X$ to a point $z_{\tau}$ in (7.1). It gives the formula
$$
\widehat{m}^{\tau}(\lambda,
b)=\widehat{m}^{\tau}(\lambda)=\varphi_{\lambda}(\exp -\tau V).
$$
The Lemma is proved.
\end{proof}

The following result is a specification of the Theorem 5.1.

\begin{thm}

There exists  a constant $ c=c(X,N)$ such that for any given
$\omega>0,  n\in \mathbb{N}\bigcup\{0\}, $ for every
$(r,N)$-lattice $Z(\{x_{\mu}\},r, N)$ with
$$
0<r<c(\omega^{2}+\|\rho\|^{2})^{-1/2}
$$
and any
\begin{equation}
 \tau<\left(\omega^{2}+\|\rho\|^{2}\right)^{-(n+1)/2},
\end{equation}
the following hold true.

 1) There exists a frame $\Phi_{j}\in B_{\omega}(X)$ such that for
 any $f\in B_{\omega}(X)$ the following reconstruction formula
 holds true

\begin{equation}
f= \sum_{j}\delta_{x_{j}}((-\Delta)^{n}m_{\tau}\ast f)\Phi_{j},
\end{equation}
which means that $f$ can be reconstructed from averages of
$(-\Delta)^{n}f$ over spheres of radius $\tau$ with centers at
$x_{j}\in Z(\{x_{\mu}\},r, N)$.

2) Every $f\in B_{\omega}(X)$ is a limit (in Sobolev and uniform
norms) of the functions
$$
\sum_{j}f\ast\phi(x_{j})\Lambda_{j}^{2^{l}d+s},
$$
when $l$ goes to infinity and
$$
\widehat{\Lambda_{j}^{2^{l}d+s}}=\frac{\widehat{L_{j}^{2^{l}d+s}}}
{(\|\lambda\|^{2}+\|\rho\|^{2})^{n}\varphi_{\lambda}(\exp\tau V)},
$$
where $\widehat{L_{j}^{2^{l}d+s}}$ is given by the formula (6.1)
\end{thm}

\begin{proof}
First we are going to show that
 the following estimate holds true
$$
|(\|\lambda\|^{2}+\|\rho\|^{2})^{n}\varphi_{\lambda}(\exp (-\tau
V))-1|\leq
$$
\begin{equation}
 \min\left\{2(\|\lambda\|^{2}+\|\rho\|^{2})^{n};
\tau^{2}(\|\lambda\|^{2}+\|\rho\|^{2})^{n+1}\right\}.
\end{equation}
Indeed, since $\varphi_{\lambda}$ is a spherical function, the
function
$$(\|\lambda\|^{2}+\|\rho\|^{2})^{n}\varphi_{\lambda}(\exp
(-\tau V))=\Phi_{\lambda}(\tau)$$
 is zonal and $\Delta \Phi_{\lambda}(\tau)$
can be calculated according to the formula (2.5). By using this
fact  along with the fact that the spherical function
$\Phi_{\lambda}$ is an eigenfunction the Laplace-Beltrami operator
$(-\Delta)$ with the eigenvalue
$(\|\lambda\|^{2}+\|\rho\|^{2})^{n+1}$ we obtain
$$
1-\Phi_{\lambda}(\tau)=
$$
$$
-\int_{0}^{\tau} \left(S(\sigma)\right)^{-1}\left(
\int_{0}^{\sigma}S(\gamma)\left(S^{-1}(\gamma)\frac{d}{d\gamma}
\left(S(\gamma)\frac{d\Phi_{\lambda}(\gamma)}{d\gamma}\right)\right)d\gamma
\right) d\sigma =
$$
$$
(\|\lambda\|^{2}+\|\rho\|^{2})^{n+1}\int_{0}^{r}\left(S(\sigma)\right)^{-1}
\left(\int_{0}^{\sigma}S(\gamma)\Phi_{\lambda}(\gamma)d\gamma\right)d\sigma.
$$

By using the following change of variables: $\sigma=s\tau,
\gamma=st\tau$, and the inequality $|\varphi_{\lambda}(g)|\leq 1$,
we obtain the estimate (7.7).

 According to inequalities (7.5) and (7.7) the
  Helgason-Fourier transform of
the distribution $(-\Delta)^{n}m_{\tau}$ does not have zeros on
the interval $[-\omega, \omega]$. Consequently, the first
statement of the Theorem    is a consequence of the Theorem 5.1.
 Note, that
since the Helgason-Fourier transform of the distribution
$(-\Delta)^{n}m_{\tau}$ is given by the formula (7.3)   the
Theorem 5.1 gives that every function $\Phi_{j}$ is the inverse
Helgason-Fourier transform of the function
$$
\frac{\widehat{\Theta_{j}}}
{(\|\lambda\|^{2}+\|\rho\|^{2})^{n}\varphi_{\lambda}(\exp\tau V)},
$$
where functions $\widehat{\Theta_{j}}$ form a frame which is dual
to the frame $\widehat{\delta_{x_{j}}}$ in the space
$L_{2}\left([-\omega,\omega]\times B; |c(\lambda)|^{-2}d\lambda
db\right).$

The second part of the Theorem is a  consequence of the Theorem
6.4.
\end{proof}

 As a
particular case with $\tau=0$ we have the so called "derivative"
sampling, which means that a function can be reconstructed from
the values of $\Delta^{n}f$ as

$$
f= \sum_{j}\delta_{x_{j}}((-\Delta)^{n}f)\Phi_{j}, n\in
\mathbb{N}\cup {0},
$$
for corresponding frame $\Phi_{j}$.

Another particular case we obtain when $n=0$. In this situation
the formula (6.6) tells that a function can be reconstructed from
its averages over spheres of radius $\tau$ with centers at
$x_{j}$.

The case $\tau=0, n=0,$ corresponds to a point-wise sampling.

 \textbf{Acknowledgement:} I would like to thank Professor A.
Kempf for inviting me to the University of Waterloo and for
interesting discussions during which I learned about  connections
of the Sampling Theory on manifolds and some ideas in Modern
Physics.

\makeatletter
\renewcommand{\@biblabel}[1]{\hfill#1.}\makeatother


\begin{thebibliography}{12}


\bibitem{B}
J.~Benedetto, {\em Irregular sampling and frames}, Academic Press,
Boston, 1992, pp. 445--507.


\bibitem{BF}
J.~Benedetto, P.~Ferreira, {\em Modern Sampling Theory},
Birkaesuer, (2001).


\bibitem{B64}
A.~Beurling, {\em Local Harmonic analysis with some applications
to differential operators}, Some Recent Advances in the Basic
Sciences, vol. 1, Belfer Grad. School Sci. Annu. Sci. Conf. Proc.,
A. Gelbart, ed., 1963-1964, 109-125.

\bibitem{BM67}
A.~Beurling and P.~Malliavin, {\em On the closure of characters
and the zeros of entire functions}, Acta Math.,\textbf{118},
(1967), 79-95.



\bibitem {DS}
R.~Duffin, A. Schaeffer, {\em A class of nonharmonic Fourier
series}, Trans. AMS, 72, (1952), 341-366.

\bibitem{FG}H.~ Feichtinger, K.~Grochenig, {\em Theory and
practice of irregular sampling. Wavelets: mathematics and
applications}, 305--363, Stud. Adv. Math., CRC, Boca Raton, FL,
1994.

\bibitem{FP1} H.~Feichtinger and I.~Pesenson,
{\em Iterative recovery of band limited functions on manifolds},
in Wavelets, Frames and  Operator Theory,(C. Heil, P.E.T. Jorgensen,
D.R. Larson, editors), Contemp. Math., 345, AMS, (2004), 137-153.




\bibitem{FP} H.~Feichtinger and I.~Pesenson,
{\em A reconstruction method for  band-limited signals on the
hyperbolic plane}, Sampl. Theory Signal Image Process. 4 (2005),
no. 2, 107--119.

\bibitem{Helg1}
S.~Helgason, {\em A duality for symmetric spaces with applications
to group representations}, Adv. Math. {\bf 5}, (1970), 1-154.

\bibitem {Helg2}
S.~Helgason, {\em Differential Geometry and Symmetric Spaces},
Academic, N.Y., 1962.

\bibitem{H2}
S.~Helgason, {\em The Abel, Fourier and Radon transforms on
symmetric spaces,} arXiv.math.RT/0506049v1 2Jun 2005.

\bibitem {K1}
A.~Kempf, {\em Aspects of information theory in curved space},
Talk presented at 10-th Conf. Gen. Rel. and Rel. Astrophysics,
Guelph,Canada, 28-31 May 2003, e-Print Archive: gr-qc/0306104.


\bibitem {K2} A.~Kempf, {\em A covariant information density cutoff in
curved space-time}, Phys. Rev. Lett. 92:221301, 2004.


\bibitem{Lan67}
H.~Landau, {\em Necessary density conditions for sampling and
interpolation of certain entire functions}, Acta. Math.,
\textbf{117}, (1967), 37-52.



\bibitem{Pes90}
 I.~Pesenson, {\em  The Bernstein Inequality in the Space of Representation
 of Lie group}, Dokl. Acad. Nauk USSR {\bf 313} (1990), 86--90;
 English transl. in Soviet Math. Dokl. {\bf 42} (1991).


\bibitem{Pes1}
 I.~Pesenson, {\em  A sampling theorem on homogeneous manifolds},
 Transactions of AMS, Vol. {\bf 352}(9),(2000), 4257-4269.



\bibitem{Pes4}
I.~Pesenson, {\em Poincare-type inequalities and reconstruction of
Paley-Wiener functions on manifolds}, J. of Geom. Anal., {\bf
4}(1), (2004), 101-121.



\bibitem{Pes8}
I.~Pesenson, {\em Sampling of Band limited vectors}, J. of Fourier
Analysis and Applications {\bf 7}(1), (2001), 93-100.



\bibitem{P10} I.~Pesenson,
{\em Frames for Paley-Wiener spaces on Riemannian manifolds}, will
appear in Integral Geometry and Tomography,(A. Markoe,  E.T.
Quinto, editors), Contemp. Math., AMS, (2006).



 \bibitem{Pes5} I. ~Pesenson, {\em Lagrangian splines, Spectral Entire
 Functions and Shannon-Whittaker Theorem on Manifolds}, Temple University
Research Report 95-87,
 (1995), 1-28.

\bibitem {SW}
C.~Shannon, W.~Weaver, {\em The Mathematical Theory of
Communication}, Univ. of Illinois Press, 1963.


\bibitem{Schoen}
I.~Schoenberg, {\em Cardinal Spline Interpolation},  CBMS, 12
SIAM, Philadelphia, 1973.

\bibitem{S}
R.~Strichartz, {\em Analysis of the Laplacian on the complete
Riemannian manifold}, J. of Funct. Anal.,{\bf 52},(1983), 48-79.
\end{thebibliography}
 \end{document}